\documentclass{article}
\usepackage[english]{babel}
\usepackage{caption,amsmath,amssymb,graphicx,enumerate,xcolor,amsthm,url,geometry}
\usepackage[colorlinks=true,bookmarksnumbered=true, pdfauthor={Bettinelli, Jacob, Miermont}]{hyperref}

\newcommand{\assign}{:=}
\newcommand{\mathd}{\mathrm{d}}
\newcommand{\nin}{\not\in}
\newcommand{\tmaffiliation}[1]{\thanks{#1}}
\newcommand{\tmmathbf}[1]{\ensuremath{\boldsymbol{#1}}}
\newcommand{\tmop}[1]{\ensuremath{\operatorname{#1}}}
\newtheorem{corollary}{Corollary}
\newtheorem{definition}{Definition}
\newtheorem{lemma}{Lemma}
\newtheorem{proposition}{Proposition}
\newtheorem{theorem}{Theorem}
\theoremstyle{remark}
\newtheorem*{remark}{Remark}
\begin{document}

\title{The scaling limit of uniform random plane maps, \textit{via} the
Ambj{\o}rn--Budd bijection}

\author{
  J{\'e}r{\'e}mie Bettinelli
	\tmaffiliation{CNRS \& Institut {\'E}lie Cartan, Nancy}
  \and
  Emmanuel Jacob
  \tmaffiliation{ENS de Lyon}
  \and
  Gr{\'e}gory Miermont
  \tmaffiliation{ENS de Lyon \& Institut Universitaire de France}
}

\maketitle

\begin{abstract}
  We prove that a uniform rooted plane map with $n$ edges converges in
  distribution after a suitable normalization to the Brownian map for the
  Gromov--Hausdorff topology. A recent bijection due to Ambj{\o}rn and Budd
  allows to derive this result by a direct coupling with a uniform random
  quadrangulation with~$n$ faces. 
\end{abstract}

\section{Introduction}

\subsection{Context}

The topic of limits of random maps has met an increasing interest over the
last two decades, as it is recognized that such objects provide natural model
of discrete and continuous 2-dimensional geometries {\cite{AmDuJo97,AnSc03}}.
Recall that a plane map is a cellular embedding of a finite graph (possibly
with multiple edges and loops) into the sphere, considered up to
orientation-preserving homeomorphisms. By \emph{cellular}, we mean that
the faces of the map (the connected components of the complement of edges) are
homeomorphic to $2$-dimensional open disks. A popular setting for studying
scaling limits of random maps is the following. We see a map $\tmmathbf{m}$
as a metric space by endowing the set $V(\tmmathbf{m})$ of its vertices with its natural graph
metric $d_{\tmmathbf{m}}$: the graph distance between two vertices is the
minimal number of edges of a path linking them. We then choose at random a map
of ``size'' $n$ in a given class and look at the limit as $n \rightarrow
\infty$ in the sense of the Gromov--Hausdorff topology \cite{gromov99} of
the corresponding metric space, once rescaled by the proper factor.

This question first arose in \cite{CSise}, focusing on the class of plane
quadrangulations, that is, maps whose faces are of degree $4$, and where the
size is defined as the number of faces. A series of papers, including
\cite{MM05,legall06,miertess,legall08,BoGu08b}, have been motivated by this
question and contributed to its solution, which was completed in
\cite{legall11,miermont11} by different approaches. Specifically, there
exists a random compact metric space~$\tmmathbf{S}$ called \emph{the Brownian
map} such that, if $Q_{n}$ denotes a uniform random (rooted) quadrangulation
with $n$ faces, then the following convergence holds in distribution for the
Gromov--Hausdorff topology on the
set of isometry classes of compact metric spaces:
\begin{equation}
  \left( \frac{9}{8n} \right)^{1/4} Q_{n} \xrightarrow[n \to \infty]{( d )}
  \tmmathbf{S}. \label{eq:convQ}
\end{equation}
Here and later in this paper, if $\mathbb{X}= ( X,d )$ is a metric
space and $a>0$, we let $a\mathbb{X}= ( X,a d )$ be the rescaled
space, and we understand a map $\tmmathbf{m}$ as the metric space $( V
( \tmmathbf{m} ) ,d_{\tmmathbf{m}} )$.

Le~Gall {\cite{legall11}} also gave a general method to prove such a limit
theorem in a broader context, that applies in particular to uniform
$p$-angulations (maps whose faces are of degree $p$) for any $p \in \{
3,4,6,8,10 \ldots \}$. When this method applies, the scaling factor $n^{-1/4}$
and the limiting metric space $\tmmathbf{S}$ are the same, only the scaling
constant $( 9/8 )^{1/4}$ may differ. One says that the Brownian map possesses
a property of universality, and one actually expects the method to work for
many more ``reasonable'' classes of maps. Roughly speaking, this approach
relies on two ingredients:
\begin{enumerate}[($i$)]
  \item A bijective encoding of the class of maps by a family of labeled trees
  that converge to the Brownian snake, in which the labels represent the
  distances to a uniform point of the map.\label{ing1}
  
  \item A property of invariance under re-rooting for the model under
  consideration and for the limiting space $\tmmathbf{S}$.\label{ing2}
\end{enumerate}
Interestingly enough, as of now, the only known method to derive the
invariance under re-rooting of the Brownian map needed in~($\ref{ing2}$)  is by using
the convergence of some root-invariant discrete model to the Brownian map, as
in (\ref{eq:convQ}). A robust and widely used bijective encoding in
obtaining~($\ref{ing1}$) is the Cori--Vauquelin--Schaeffer bijection
{\cite{CoVa,schaeffer98}} and its generalization by
Bouttier--Di~Francesco--Guitter {\cite{BdFGmobiles}}, see for instance
{\cite{MaMi07,mierinv}}. However, this bijection becomes technically uneasy to
manipulate when dealing with non-bipartite maps (with the notable exception of
triangulations) or maps with topological constraints. Recently, Addario-Berry
and Albenque {\cite{addarioalbenque2013simple}} obtained the convergence to
the Brownian map for the classes of simple triangulations and simple quadrangulations (maps without loops or multiple edges), by using another
bijection due to Poulalhon and Schaeffer {\cite{PoSc06}}.

In the present paper, we continue this line of research with another
fundamental class of maps, namely uniform random plane maps with a prescribed
number of edges. The key to our study is to use a combination of the
Cori--Vauquelin--Schaeffer bijection, together with a recent bijection due to
Ambj{\o}rn and Budd {\cite{ambjornbudd}}, that allows to couple directly a
uniform (pointed) map with $n$ edges and a uniform quadrangulation with $n$
faces, while preserving distances asymptotically. This allows to transfer
known results from uniform quadrangulations to uniform maps, in a way that is
comparatively easier than a method based on the
Bouttier--Di~Francesco--Guitter bijection. However, and this was a bit of a
surprise to us, proving the appropriate re-rooting invariance necessary to
apply~($\ref{ing2}$) above does require some substantial work.

We note that our results answer a question asked in the very recent preprint
{\cite{BoFuGu13}}. Let us also mention that, in parallel to our work,
C{\'e}line Abraham {\cite{abr14}} has obtained a similar result to ours for
uniform bipartite maps, by using an approach based on the
Bouttier--Di~Francesco--Guitter bijection.

\subsection{Main results}

We need to introduce some notation and terminology at this point. If $e$ is an
oriented edge of a map, the face that lies to the left of $e$ will be called
the face \emph{incident} to $e$. We denote by $e^{-}$, $e^{+} $ and
$\tmop{rev} ( e )$ the origin, end and reverse of the oriented edge $e$. It
will be convenient to consider \emph{rooted} maps, that is, maps given
with a distinguished oriented edge called the \emph{root}, and usually
denoted by~$e_{\ast}$. The {\emph{root vertex}} is by definition the vertex~$e_{\ast}^{-}$.

We let $\mathcal{M}_{n}$ be the set of rooted plane maps with $n$ edges, and
$\mathcal{M}_{n}^{\bullet}$ be the set of rooted and \emph{pointed} plane
maps with $n$ edges, i.e., of pairs $( \tmmathbf{m},v_{\ast} )$ where
$\tmmathbf{m} \in \mathcal{M}_{n}$ and $v_{\ast}$ is a distinguished element
of $V ( \tmmathbf{m} )$.

Similarly, we let $\mathcal{Q}_{n}$ (resp.\ $\mathcal{Q}_{n}^{\bullet}$) be
the set of rooted (resp.\ rooted and pointed) quadrangulations with $n$ faces.
We also let $\mathbb{T}_{n}$ be the set of well-labeled trees with $n$ edges,
i.e.,~of pairs $( \tmmathbf{t},\tmmathbf{l} )$ where~$\tmmathbf{t}$ is a
rooted plane tree with $n$ edges, and $\tmmathbf{l}$ is an integer-valued
label function on the vertices of~$\tmmathbf{t}$ that assigns the value $0$ to
the root vertex of $\tmmathbf{t}$, and such that $| \tmmathbf{l} ( u )
-\tmmathbf{l} ( v ) | \leqslant 1$ whenever $u$ and $v$ are neighboring
vertices in~$\tmmathbf{t}$.

There exists a well-known correspondence, sometimes called the
\emph{trivial bijection}, between the sets $\mathcal{M}_{n}$ and
$\mathcal{Q}_{n}$. Starting from a rooted map $\tmmathbf{m}$, we
add a vertex inside each face of $\tmmathbf{m}$, and join this vertex to every
corner of the corresponding face by a family of non-crossing arcs. If we
remove the relative interiors of the edges of $\tmmathbf{m}$, then the map
formed by the added arcs is a quadrangulation~$\tmmathbf{q}$, which we can
root in a natural way from the root of $\tmmathbf{m}$ by fixing some
convention. In this construction, the set of vertices of $\tmmathbf{m}$ is
exactly the set $V_{0} ( \tmmathbf{q} )$ of vertices of $\tmmathbf{q}$ that
are at even distance from the root vertex: this comes from the natural
bipartition $V_{0} ( \tmmathbf{q} ) \sqcup V_{1} ( \tmmathbf{q} )$ of $V (
\tmmathbf{q} )$ given by the vertices of $\tmmathbf{m}$ and the vertices that
are added in the faces of $\tmmathbf{m}$.

However, the graph distances in $\tmmathbf{m}$ and those in $\tmmathbf{q}$
are not related in an obvious way, except for the elementary bound
\[ d_{\tmmathbf{q}} ( u,v ) \leqslant 2d_{\tmmathbf{m}} ( u,v ) \leqslant
   \frac{\Delta ( \tmmathbf{m} )}{2} \, d_{\tmmathbf{q}} ( u,v )  
   \qquad u,v \in V ( \tmmathbf{m} ) =V_{0} ( \tmmathbf{q} )  ,
\]
where $\Delta ( \tmmathbf{m} )$ denotes the largest degree of a face
in~$\tmmathbf{m}$. On the other hand, it was noticed recently by Ambj{\o}rn
and Budd {\cite{ambjornbudd}} that there exists another natural bijection
between $\mathcal{M}_{n}^{\bullet} \times \{ 0,1 \}$ and
$\mathcal{Q}_{n}^{\bullet}$, which is much more faithful to graph distances.
This bijection is constructed in a way that is very similar to the well-known
Cori--Vauquelin--Schaeffer ($\tmop{CVS}$) bijection between
$\mathcal{Q}_{n}^{\bullet}$ and $\mathbb{T}_{n} \times \{ 0,1
\}$, and is in some sense dual to it. For the reader's convenience, we will
introduce the two bijections simultaneously in Section \ref{sec:CVSAB}.

The Ambj{\o}rn--Budd ($\tmop{AB}$) bijection provides a natural
coupling between a uniform random element $( Q_{n} ,v_{\ast} )$ of
$\mathcal{Q}_{n}^{\bullet}$, and a uniform random element $(
M_{n}^{\bullet} ,v_{\ast} )$ of $\mathcal{M}_{n}^{\bullet}$. Using
this coupling, it was observed already \cite{ambjornbudd,BoFuGu13}
that the ``two-point functions'' that govern the limit
distribution of the distances between two uniformly chosen points in
$M_{n}^{\bullet}$ and $Q_{n}$ coincide. In this paper, we show that
much more is true. 
\begin{theorem}
  \label{sec:main-results}
  Let $( Q_{n} ,v_{\ast} )$ and $(M_{n}^{\bullet} ,v_{\ast} )$ be
  uniform random elements of $\mathcal{Q}_{n}^{\bullet}$ and
  $\mathcal{M}_{n}^{\bullet}$ respectively, that are in correspondence
  {\it via} the Ambj{\o}rn--Budd bijection. Then we have the following joint convergence in
distribution for the Gromov--Hausdorff topology
$$  \left(\! \left( \frac{9}{8n} \right)^{1/4} M_{n}^{\bullet} , \left(
  \frac{9}{8n} \right)^{1/4} Q_{n} \right) \underset{n \rightarrow
  \infty}{\xrightarrow{( d )}} ( \tmmathbf{S},\tmmathbf{S} )  ,$$
where $\tmmathbf{S}$ is the Brownian map.
\end{theorem}

%with this coupling, we actually have
%\begin{equation}
 % \mathd_{\tmop{GH}} ( M_{n}^{\bullet} ,Q_{n} ) =o_{P} ( n^{1/4} )
 % \label{eq:dmq}
%\end{equation}%
%where $\mathd_{\tmop{GH}}$ is the Gromov--Hausdorff distance on the
%set of isometry classes of compact metric spaces.  Also, for a
%sequence of random variables $( X_{n} ,n \geqslant 1 )$ and a
%deterministic positive sequence $( a_{n} ,n \geqslant 1 )$, we let
%$X_{n} =o_{P} ( a_{n} )$ if $X_{n} /a_{n}$ converges to $0$ in
%probability. We similarly let $X_{n} =O_{P} ( a_{n} )$ to mean that
%the family of laws of the random variables $X_{n} /a_{n}$ is
%tight. 

A very striking aspect of this is that the scaling constant $( 9/8
)^{1/4}$ is the same for $M_{n}^{\bullet}$ and for $Q_{n}$. This
implies in particular that
$$n^{-1/4}\mathd_{\tmop{GH}} ( M_{n}^{\bullet} ,Q_{n} ) \xrightarrow[n\to\infty]{P}0$$
where $\mathd_{\tmop{GH}}$ is the Gromov--Hausdorff distance beween
two compact metric spaces, which, to paraphrase
the title of {\cite{marckert04}}, says that ``the
$\tmop{AB}$ bijection is asymptotically an isometry.'' Although obtaining this scaling
constant is theoretically possible using the methods of
{\cite{mierinv}}, the computation would be rather involved.

At the cost of an extra ``de-pointing lemma,'' (Proposition \ref{TV}) this will imply the following
result.

\begin{corollary}
  \label{principal}Let $M_{n}$ be a uniformly distributed random variable in
  $\mathcal{M}_{n}$. The following convergence in distribution holds for the
  Gromov--Hausdorff topology
  \[ \left( \frac{9}{8n} \right)^{1/4} M_{n} \xrightarrow[n \to \infty]{( d
     )} \tmmathbf{S} \]
  where $\tmmathbf{S}$ is the Brownian map. 
\end{corollary}

As was pointed to us by {\'E}ric Fusy, it is likely that our methods
can also be used to prove convergence of uniform (pointed) bipartite
maps with $n$ edges. Indeed, following {\cite{BoFuGu13}}, these are in
natural correspondence \textit{via} the $\tmop{AB}$ bijection with
pointed quadrangulations with no confluent faces (see below for
definitions). In turn, the latter are in correspondence \textit{via}
the $\tmop{CVS}$ bijection with ``very well-labeled trees,'' which are
elements of $\mathbb{T}_{n}$ in which the labels of two neighboring
vertices differ by exactly $1$ in absolute value (this has the effect
of replacing the scaling constant $( 9/8 )^{1/4}$ with
$2^{-1/4}$). However, checking the details of this approach still
requires some work, and we did not pursue this to keep the length of
this paper short, and because this result has already been obtained by
Abraham {\cite{abr14}} using a more ``traditional'' and robust
bijective method.

In Section \ref{sec:CVSAB}, we present the two abovementionned
bijections.  Section \ref{secpoint} is devoted to the comparison
between the distributions of $M_{n}$ and $M_{n}^{\bullet} $. Section
\ref{secenc} is dedicated to the heart of the proof of Theorem
\ref{sec:main-results}, and Section \ref{reroot} proves the key
re-rooting identity~($\ref{ing2}$).

\paragraph*{Acknowledgments} We thank J{\'e}r{\'e}mie Bouttier and {\'E}ric Fusy for
stimulating discussions and insights during the elaboration of this work.

\section{Cori--Vauquelin--Schaeffer and Ambj{\o}rn--Budd
bijections}\label{sec:CVSAB}

In most of this section, we fix an element $( \tmmathbf{q},v_{\ast} ) \in
\mathcal{Q}_{n}^{\bullet}$, and consider one particular embedding of
$\tmmathbf{q}$ in the plane. We label the elements of $V ( \tmmathbf{q} )$ by
their distance to $v_{\ast}$, hence letting $\tmmathbf{l}_{+} ( v )
=d_{\tmmathbf{q}} ( v,v_{\ast} )$. Using the bipartite nature of
quadrangulations, each quadrangular face is of either one of two types, which
are called \emph{simple} and \emph{confluent}, depending on the
pattern of labels of the incident vertices. This is illustrated on
Figure~\ref{fig:convbij}, where the four edges incident to a face of
$\tmmathbf{q}$ are represented in thin black lines, and the four corresponding
vertices are indicated together with their respective labels. The
Cori--Vauquelin--Schaeffer ($\tmop{CVS}$) bijection consists in adding one extra
``red'' arc inside each face, linking the vertex with largest label of simple
faces to the next one in the face in clockwise order, and the two vertices
with larger label in confluent faces. The Ambj{\o}rn--Budd ($\tmop{AB}$) bijection
adopts the opposite rules, adding the ``green'' arcs to $\tmmathbf{q}$. \

\begin{figure}[ht]
  \centering\includegraphics{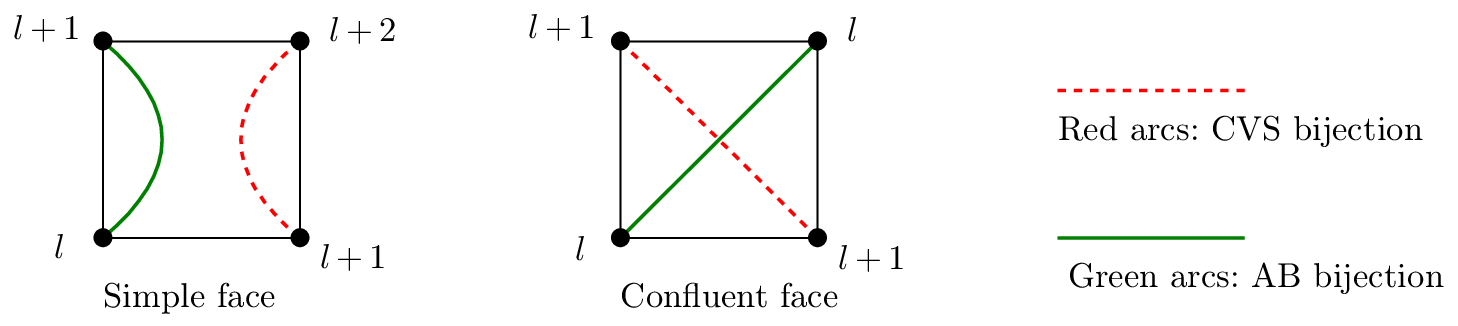}
  \caption{Convention for adding arcs in the bijections.}
	\label{fig:convbij}
\end{figure}

The connected graphs whose edge-sets are formed by the arcs of either color
(red or green) are obviously embedded graphs.

\paragraph{The Cori--Vauquelin--Schaeffer bijection}

For the $\tmop{CVS}$ bijection, the ``red'' embedded graph is a plane tree
$\tmmathbf{t}$ with $n$ edges, with vertex set $V ( \tmmathbf{t} ) =V (
\tmmathbf{q} ) \setminus \{ v_{\ast} \}$. This tree also inherits a label
function, which is simply the label function $\tmmathbf{l}_{+}$ above. It also
inherits a root from the root $e_{\ast}$ of $\tmmathbf{q}$, following a
convention that we will not need to describe in detail. What is important
about the rooting convention, however, is the following. If we are given a
vertex~$v$ and an oriented edge~$e$ in~$\tmmathbf{q}$, we say that $e$
{\emph{points towards}} $v$ if $d_{\tmmathbf{q}} ( e^{+} ,v ) =
d_{\tmmathbf{q}} ( e^{-} ,v ) -1$. Then the root vertex of $\tmmathbf{t}$ is
equal to $e_{\ast}^{-}$ if $e_{\ast}$ points towards $v_{\ast}$, and to
$e_{\ast}^{+}$ otherwise. We let $\epsilon$ be respectively equal to $0$ or
$1$ depending on which of these two situations occur.

\begin{remark}
  Throughout this work, we will never consider the root (edge) of the tree;
  only its {\emph{root vertex}}, that is, the origin of its root, will be of
  importance.
\end{remark}

It is then usual to define the label function $\tmmathbf{l} ( v )
=\tmmathbf{l}_{+} ( v ) -\tmmathbf{l}_{+} ( \tmop{root} ( \tmmathbf{t} ) )
=\tmmathbf{l}_{+} ( v ) -\tmmathbf{l}_{+} ( e^{-}_{\ast} ) - \epsilon$, with
values in $\mathbb{Z}$.

\begin{proposition}
  \label{CVS}The mapping $\tmop{CVS} :\mathcal{Q}_{n}^{\bullet} \rightarrow
  \mathbb{T}_{n} \times \{ 0,1 \}$ sending the pointed quadrangulation $(
  \tmmathbf{q},v_{\ast} )$ to the pair $( ( \tmmathbf{t},\tmmathbf{l} ) ,
  \epsilon )$ as above, is a bijection. 
\end{proposition}

In the following, we will often omit $\epsilon$ from the notation, and will
only refer to it when it plays an indispensable role.

\paragraph{The Ambj{\o}rn--Budd bijection}

On the other hand, the ``green'' embedded graph formed following the rules of
the $\tmop{AB}$ bijection is a plane map $\tmmathbf{m}$ with $n$ edges, but with
vertex-set equal to $V ( \tmmathbf{q} ) \setminus V_{\max} ( \tmmathbf{q} )$,
where $V_{\max} ( \tmmathbf{q} )$ is the set of vertices $v$ of $\tmmathbf{q}$
that are local maxima of the function $\tmmathbf{l}_{+}$, i.e.,~such that
$d_{\tmmathbf{q}} ( u,v_{\ast} ) =d_{\tmmathbf{q}} ( v,v_{\ast} ) -1$ for
every neighbor $u$ of $v$. Note that $V_{\max} ( \tmmathbf{q} )$ really
depends on the pointed map $( \tmmathbf{q},v_{\ast} )$ rather than on
$\tmmathbf{q}$ alone, but we nevertheless adopt this shorthand notation for
convenience. One should note that the distinguished vertex $v_{\ast} \in V (
\tmmathbf{q} )$ is never a local maximum of $\tmmathbf{l}$ (it is indeed the
global minimum!), so that it is an element of $V ( \tmmathbf{m} )$, also
naturally distinguished.

By the Euler formula, this implies that $\tmmathbf{m}$ has $\#V_{\max} (
\tmmathbf{q} )$ faces. One can be more precise by saying that when embedding
$\tmmathbf{m}$ and $\tmmathbf{q}$ jointly in the plane as in the above
construction, each face of $\tmmathbf{m}$ contains exactly one of the vertices
of $V_{\max} ( \tmmathbf{q} )$. Finally, we can use the root $e_{\ast}$ of
$\tmmathbf{q}$ to root the map $\tmmathbf{m}$ according to some convention
that we will not describe fully, but for which the root vertex of
$\tmmathbf{m}$ is equal to $e_{\ast}^{+}$ if $e_{\ast}$ points towards
$v_{\ast}$, and to $e_{\ast}^{-}$ otherwise. We let $\epsilon$ be equal to 0
or 1 accordingly. See Figure~\ref{fig:notation} for an example of both
bijections.

\begin{proposition}
  \label{AB}The mapping $\tmop{AB} :\mathcal{Q}_{n}^{\bullet} \rightarrow
  \mathcal{M}_{n}^{\bullet} \times \{ 0,1 \}$ sending the pointed
  quadrangulation $( \tmmathbf{q},v_{\ast} )$ to the pair $( (
  \tmmathbf{m},v_{\ast} ) , \epsilon )$ as above, is a bijection. 
\end{proposition}

Again, we will usually omit $\epsilon$ from the notation. The map
$\tmmathbf{m}$ also inherits the labeling function~$\tmmathbf{l}_{+}$ from the
quadrangulation $\tmmathbf{q}$, but contrary to what happens for the $\tmop{CVS}$
bijection, this information turns out to be redundant thanks to the remarkable
identity
\begin{equation}
  d_{\tmmathbf{m}} ( v,v_{\ast} ) =d_{\tmmathbf{q}} ( v,v_{\ast} )
  =\tmmathbf{l}_{+} ( v ) , \qquad v \in V ( \tmmathbf{m} ) =V (
  \tmmathbf{q} ) \setminus V_{\max} ( \tmmathbf{q} ) . \label{eq:samegeod}
\end{equation}
In fact, we are going to make this identity slightly more precise by showing
that $\tmmathbf{q}$ and $\tmmathbf{m}$ actually ``share'' some specific
geodesics to $v_{\ast}$. In order to specify the exact meaning of this, we
need a couple extra definitions. Let $e$ be an oriented edge in
$\tmmathbf{q}$, and let $f$ be the face incident to~$e$. We say that $e$ is
\emph{special} if the green arc associated with $f$ by the $\tmop{AB}$ bijection is
incident to the same two vertices as $e$ (in particular, $f$ must be a simple
face). In this case, we let $\tilde{e}$ be this green arc. On the above
picture of a simple face, the face is incident to exactly one special edge,
which is the one on the left, oriented from top to bottom. More generally, we
use the following definition:

\begin{definition}
  \label{running}If $c= ( e_{1} ,e_{2} , \ldots ,e_{k} )$ is a chain of
  oriented edges in $\tmmathbf{q}$, in the sense that $e_{i}^{+} =e_{i+1}^{-}$
  for every $i \in \{ 1,2 \ldots ,k-1 \}$, and if all these oriented edges are
  special, then we say that the chain $c$ is special and we let $\tilde{c} = (
  \tilde{e}_{1} , \ldots , \tilde{e}_{k} )$ be the corresponding chain in
  $\tmmathbf{m}$. 
\end{definition}

Next, if $e$ is an edge of $\tmmathbf{q}$, we can canonically give it an
orientation so that it points towards~$v_{\ast}$. Then, among all geodesic
chains $( e,e_{1,} \ldots ,e_{k} )$ from $e^{-}$ to $v_{\ast}$ with first step
$e$ (so that $k=d_{\tmmathbf{q}} ( e^{-} ,v_{\ast} ) -1$), there is a
distinguished one, called the \emph{left-most geodesic to $v_{\ast}$ with
first step $e$}, which is the one for which the clockwise angular sector
between $e_{i}$ and $e_{i+1}$, and excluding $e_{i+1}$, contains only edges
pointing towards $e_{i}^{+} =e_{i-1}^{-}$, with the convention that $e_{0}
=e$. We let $\gamma ( e )$ be this distinguished geodesic, and $\hat{\gamma} (
e ) = ( e_{1} ,e_{2} , \ldots ,e_{k} )$ be the same path, with the first step
removed. This is illustrated in the following picture, where two corresponding
steps of the geodesic $\gamma ( e )$ are depicted.

\begin{figure}[ht]
  \centering\includegraphics{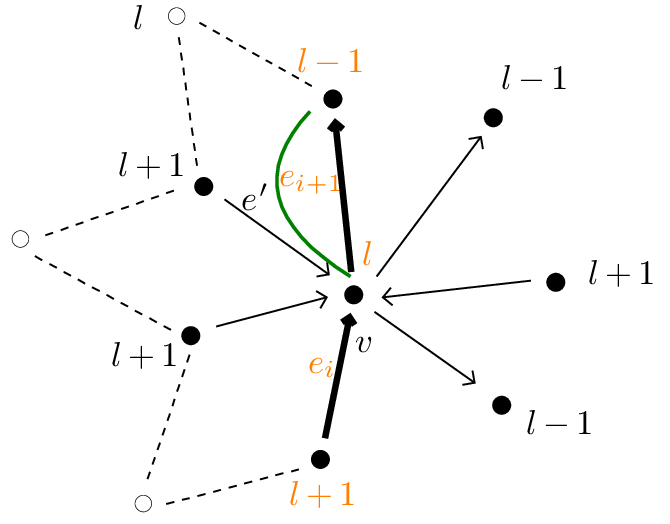}
  \caption{Two consecutive steps of a left-most geodesic.}
	\label{fig:lmg}
\end{figure}

\begin{proposition}
  \label{special}Let $e$ be an oriented edge of $\tmmathbf{q}$ that points
  toward $v_{\ast}$. Then the chain $\hat{\gamma} ( e )$ is special. 
\end{proposition}

\begin{proof}
  Here the reader might want to use Figure~\ref{fig:lmg} to follow the details
  of this proof. Fix $i \in \{ 0,1, \ldots ,k-1 \}$ and let $v=e_{i}^{+}
  =e_{i+1}^{-}$.
  
  Consider the last edge $e'$ before $e_{i+1}$ in clockwise order around $v$.
  Then by definition of the left-most geodesic, $e'$ must be pointing towards
  $v$. Then the face incident to $v$ that has the sector between $e'$ and
  $e_{i+1}$ as a corner is necessarily a simple face, and the vertex of this
  face that is diagonally opposed to $v$ must have label equal to the label
  $l=d_{\tmmathbf{q}} ( v,v_{\ast} )$ of $v$ (since the other two labels must
  be $l+1=d_{\tmmathbf{q}} ( (  e' )^{-} ,v_{\ast}$) and
  $l-1=d_{\tmmathbf{q}} ( e_{i+1}^{+} ,v )$. Therefore, $e_{i+1}$ is the
  special edge incident to this simple face.
  
  Since by hypothesis $e=e_{0}$ is pointing towards $v_{\ast}$, this implies
  by our argument that $e_{1}$ is special, and we can conclude by an induction
  argument. 
\end{proof}

Proposition \ref{special} has an apparently anecdotal consequence on which, in
fact, most of this work relies. Let $e$, $e'$ be two oriented edges of
$\tmmathbf{q}$ pointing towards $v_{\ast}$. The two left-most geodesics
$\gamma ( e ) = ( e_{0} ,e_{1} , \ldots ,e_{k} )$ and $\gamma ( e' ) = (
e'_{0} ,e'_{1} , \ldots ,e'_{k'} )$ share a maximal common suffix, say
$e_{k-r+1} =e'_{k' -r+1} , \ldots ,e_{k} =e'_{k'}$ where $r \geqslant 0$ is
the largest possible. (Note that $r$ may be equal to $k+1$ or $k' +1$, in the
case where one geodesic is entirely a suffix of the other one.) But then it
always holds that $e_{k-r}^{+} = ( e'_{k' -r} )^{+}$, so that the sequence $(
e_{0} ,e_{1} , \ldots ,e_{k-r} , \tmop{rev} ( e'_{k' -r} )  
 , \tmop{rev} ( e'_{k' -r-1} ) , \ldots , \tmop{rev} ( e'_{1} ) ,
\tmop{rev} ( e'_{0} ) )$ is a chain, with total length that we denote by
$d^{\circ}_{\tmmathbf{q}} ( e,e' )$. Recall that $\Delta ( \tmmathbf{m} )$
denotes the largest face degree of~$\tmmathbf{m}$.

\begin{corollary}
  \label{majoration}Let $v$, $v' \in V ( \tmmathbf{m} ) =V ( \tmmathbf{q} )
  \setminus V_{\max} ( \tmmathbf{q} )$ be given, and let $e$, $e'$ be two
  oriented edges in $\tmmathbf{q}$ both pointing towards $v_{\ast}$, and such
  that $e^{-} =v  $ and $( e' )^{-} =v'$. Then it holds that
  \[ d_{\tmmathbf{m}} ( v,v' ) \leqslant d^{\circ}_{\tmmathbf{q}} ( e,e' ) +
     \Delta ( \tmmathbf{m} ) . \]
\end{corollary}

\begin{proof}
  We assume that $e \neq e'$ to avoid trivialities. By Proposition
  \ref{special}, the geodesics $\hat{\gamma} ( e )$ and $\hat{\gamma} ( e' )$
  are special, so that there are paths in $\tmmathbf{m}$ starting from $e^{+}$
  and $( e' )^{+}$ with edges $( \tilde{e}_{1} , \ldots , \tilde{e}_{k} )$ and
  $( \tilde{e}'_{1} , \ldots , \tilde{e}'_{k'} )$ respectively. But then the
  maximal suffix shared by these paths has the same length as the one shared
  by $\gamma ( e )$ and $\gamma ( e' )$. Therefore, we can join $e^{+}$ and $(
  e' )^{+}$ in $\tmmathbf{m}$ by a path of length $d^{\circ}_{\tmmathbf{q}}
  ( e,e' ) -2$. Now by construction of the $\tmop{AB}$ bijection, the edge $e$ lies in
  a single face of $\tmmathbf{m}$, so that we can join $e^{-}$ to $e^{+}$ with
  a path of length at most $\Delta ( \tmmathbf{m} ) /2$. The same is true for
  the extremities of~$e'$, which allows to conclude. 
\end{proof}

\section{Comparing pointed and non-pointed maps}\label{secpoint}

Let $M_{n}$ be a uniformly distributed random variable in $\mathcal{M}_{n}$,
and let $( M_{n}^{\bullet} ,v_{\ast} )$ be a uniformly distributed random
variable in $\mathcal{M}_{n}^{\bullet}$. The superscript in $M_{n}^{\bullet}$
is here to indicate that, even after forgetting the distinguished vertex
$v_{\ast}$, it does not have same distribution as $M_{n}$. Rather, it holds
that
\begin{equation}
  P ( M_{n}^{\bullet} =\tmmathbf{m} ) = \frac{\#V ( \tmmathbf{m} )}{\#
  \mathcal{M}^{\bullet}_{n}}  , \qquad \tmmathbf{m}
  \in \mathcal{M}_{n}  . \label{eq:loipoint}
\end{equation}
Note that, by contrast, if $( Q_{n} ,v_{\ast} )$ is a uniformly distributed
random variable in $\mathcal{Q}_{n}^{\bullet}$, then $Q_{n}$ is indeed uniform
in $\mathcal{Q}_{n}$ since a quadrangulation with $n$ faces has $n+2 $
vertices, so that pointing such a quadrangulation does not introduce a bias.
The goal of this subsection is to obtain the following comparison theorem for
the laws of $M_{n}$ and $M_{n}^{\bullet}$. Let $\mu_{n}$ be the law of $M_{n}$
and $\mu_{n}^{\bullet}$ be the law of~$M_{n}^{\bullet}$. We let $\| \cdot \|$
denote the total variation norm of signed measures.

\begin{proposition}
  \label{TV}It holds that $\| \mu_{n} - \mu_{n}^{\bullet} \| \rightarrow 0$ as
  $n \rightarrow \infty$.
\end{proposition}

\begin{proof}
  By (\ref{eq:loipoint}), one has
  \[ \| \mu_{n} - \mu_{n}^{\bullet} \| = \sum_{\tmmathbf{m} \in
     \mathcal{M}_{n}} \left| \frac{1}{\# \mathcal{M}_{n}} - \frac{\#V (
     \tmmathbf{m} )}{\# \mathcal{M}_{n}^{\bullet}} \right|  . \]
  Now recall that
  \[ \# \mathcal{M}_{n} =\# \mathcal{Q}_{n} = \frac{2}{n+2} 
     \frac{3^{n}}{n+1} \binom{2n}{n} , \qquad \#
     \mathcal{M}_{n}^{\bullet} = \frac{1}{2} \# \mathcal{Q}_{n}^{\bullet} =
     \frac{3^{n}}{n+1}   \binom{2n}{n}  , \]
  where we used the trivial graph bijection between a rooted map with $n$
  edges and a rooted quadrangulation with $n$ faces on the one hand, and the
  $\tmop{AB}$ bijection on the other hand. This implies that
  \begin{equation}
    \| \mu_{n} - \mu_{n}^{\bullet} \| = \frac{1}{\# \mathcal{M}_{n}}
    \sum_{\tmmathbf{m} \in \mathcal{M}_{n}} \left| \frac{2\, \#V ( \tmmathbf{m}
    )}{n+2} -1 \right|  =E \left[ \left| \frac{2\, \#V ( M_{n} )}{n+2}
    -1 \right| \right]  . \label{eq:TV}
  \end{equation}

  To show that this vanishes as $n \rightarrow \infty$, we compute the first
  two moments of $\#V ( M_{n} )$. Note that by the trivial graph bijection,
  $\#V ( M_{n} )$ has same distribution as the number of vertices at even
  distance from the root vertex $e_{\ast}^{-}$ in a uniform rooted
  quadrangulation $Q_{n}$. By an obvious symmetry argument, this
  implies that
  \begin{equation}
    E \big[ \#V ( M_{n} ) \big] = \frac{1}{2} E \big[ \#V ( Q_{n} ) \big] = \frac{n+2}{2}. \label{eq:3}
  \end{equation}
  For the second moment, we use the $\tmop{CVS}$ bijection again. Select a uniform
  random vertex $v_{\ast}$ among the $n+2$ elements of $V ( Q_{n} )$ and let
  $( ( T_{n} , \ell_{n} ) , \epsilon ) = \tmop{CVS} ( Q_{n} ,v_{\ast} )$.
  Since $\ell_{n} ( v ) =d_{Q_{n}} ( v,v_{\ast} ) -d_{Q_{n}} ( e_{\ast}^{-}
  ,v_{\ast} ) - \epsilon$ for every $v \in V ( Q_{n} )$, we have that the
  vertices $v$ at even distance from $e_{\ast}^{-}$ are those for which
  $\ell_{n} ( v ) + \epsilon$ is even. So
  \begin{align*}
    E \big[ \#V ( M_{n} )^{2} \big] & = E \left[ \sum_{u,v \in V ( T_{n} )
    \cup \{ v_{\ast} \}} \mathbf{1}_{\{ \ell_{n} ( u ) + \epsilon \equiv
    \ell_{n} ( v ) + \epsilon \equiv 0  [ \tmop{mod}  2 ] \}} \right]\\
    & = ( n+2 )^{2}\, P \big( \ell_{n} ( U ) + \epsilon \equiv \ell_{n} ( V ) +
    \epsilon \equiv 0  [ \tmop{mod}  2 ] \big) .
  \end{align*}
  where $U,V$ are uniformly chosen in $V ( T_{n} ) \cup \{ v_{\ast} \}$
  conditionally given $T_{n}$ and independently of $( \ell_{n} , \epsilon )$.
  
  Plainly, the probability under consideration is equivalent to the same
  quantity where $U,V$ are instead chosen uniformly in $V ( T_{n} )$.
  Furthermore, conditionally given $T_{n} ,U,V$, the labels along the branch
  from $U$ to $V$ in $T_{n}$ form a random walk with i.i.d.\ steps that are
  uniform in $\{ -1,0,1 \}$, and thus the parity of the labels follow an
  irreducible Markov chain with values in \{0,1\} with transition matrix
  $\begin{pmatrix}
    1/3 & 2/3\\
    2/3 & 1/3\\
  \end{pmatrix}$ and stationary measure $( 1/2,1/2 )$. It follows that
  the probability that $\ell_{n} ( U )$ and $\ell_{n} ( V )$ have same parity
  is a function of $d_{T_{n}} ( U,V )$ with limit $1/2$ at infinity, while the
  probability that $\ell_{n} ( U ) + \epsilon$ is even is exactly $1/2$ since
  $\epsilon$ is a $\tmop{Bernoulli}(1/2)$ random variable independent of $( T_{n} ,
  \ell_{n} ,U,V )$. On the other hand, it is classical that $d_{T_{n}} ( U,V )
  / \sqrt{2n}$ converges to a Rayleigh distribution as $n \rightarrow \infty$,
  so that $d_{T_{n}} ( U,V )$ converges to $\infty$ in probability. These
  facts easily entail that $P ( \ell_{n} ( U ) + \epsilon \equiv \ell_{n} ( V
  ) + \epsilon \equiv 0  [ \tmop{mod}  2 ] )$ converges to $1/4$ as $n
  \rightarrow \infty$. Consequently,
  \begin{equation}
    E \big[ \#V ( M_{n} )^{2} \big] = \frac{n^{2}}{4} \big( 1+o ( 1 ) \big)  ,
    \qquad \text{as }  n \rightarrow \infty  . \label{eq:4}
  \end{equation}
  Together, equations (\ref{eq:3}) and (\ref{eq:4}) imply that $2 \, \#V ( M_{n} ) /n$
  converges to $1$ in $L^{2}$, which entails the result by~(\ref{eq:TV}).
\end{proof}

From this, we deduce a bound in probability for $\Delta ( M_{n}^{\bullet} )$.
Theorem~3 of Gao and Wormald {\cite{gaowor00}} shows that if $\Delta^{V} (
M_{n} )$ denotes the largest degree of a vertex in $M_{n}$, then
\[ P ( \Delta^{V} ( M_{n} ) > \log  n ) \rightarrow 0 \qquad
   \text{as } n \rightarrow \infty    . \]
From the obvious fact that the dual map of $M_{n}$ has same distribution as
$M_{n}$, the same is true if we replace $\Delta^{V} ( M_{n} )$ with
$\Delta ( M_{n} )$. By Proposition \ref{TV} we conclude that the same holds
for $M_{n}^{\bullet}$. 

\begin{lemma}
  \label{oplogn}It holds that as $n \rightarrow \infty$,
  \[ P(\Delta ( M_{n}^{\bullet} )> \log  n )\to 0  . \]
\end{lemma}

\section{Encoding with processes and convergence results}\label{secenc}

We now proceed by following the general approach introduced by Le Gall
{\cite{legall06,legall11}}, which we mentioned in the Introduction. It first
requires to code maps with stochastic processes. Let $( Q_{n} ,v_{\ast} )$ be
a uniform random element of $\mathcal{Q}_{n}^{\bullet}$, and let $ ( (
T_{n} , \ell_{n} ) , \epsilon ) = \tmop{CVS} ( Q_{n} ,v_{\ast} )$ and $( (
M_{n}^{\bullet} ,v_{\ast} ) , \epsilon ) = \tmop{AB} ( Q_{n} ,v_{\ast} )$.
Since $\tmop{CVS}$ and $\tmop{AB}$ are bijections, the random variables $(
T_{n} ,l_{n} )$ and $( M_{n}^{\bullet} ,v_{\ast} )$ are respectively uniform
in $\mathbb{T}_{n}$ and $\mathcal{M}_{n}^{\bullet}$, while
$\epsilon$ is uniform in \{0,1\} and independent of $( T_{n} , \ell_{n} )$ and
$( M_{n}^{\bullet} ,v_{\ast} )$. Note that our conventions imply that the
variable $\epsilon$ is indeed the same in the images by the two bijections.

\subsection{Coding with discrete processes}

For $i \in \{ 0,1, \ldots ,2n \}$ we let $c_{i}$ be the $i$-th corner of
$T_{n}$ in contour order, starting from the root corner, so in particular
$c_{0} =c_{2n}$. We extend this to a sequence $( c_{i} ,i \in \mathbb{Z} )$
by $2n$-periodicity. Let also $v_{i}$ be the vertex of $T_{n}$ that is
incident to $c_{i}$. The contour and label functions of $( T_{n,} \ell_{n} )$
are defined by
\[ C_{n} ( i ) =d_{T_{n}} ( v_{i} ,v_{0} )  , \qquad L_{n} ( i )
   = \ell_{n} ( v_{i} )  , \qquad i \in \{ 0,1, \ldots ,2n \}
    , \]
and these functions are extended to continuous functions $[ 0,2n ] \rightarrow
\mathbb{R}$ by linear interpolation between integer coordinates. Now recall
that the sets $V ( T_{n} )$ and $V ( Q_{n} ) \setminus \{ v_{\ast} \}$ are
identified by the $\tmop{CVS}$ bijection, so that we can view $v_{i}$, $0 \leqslant i
\leqslant 2n$ as elements of $V ( Q_{n} )$. With this identification we let
\[ D_{n} ( i,j ) =d_{Q_{n}} ( v_{i} ,v_{j} )  , \qquad i,j \in
   \{ 0,1, \ldots ,2n \}  , \]
and we extend $D_{n}$ to a continuous function $[ 0,2n ]^{2} \rightarrow
\mathbb{R}$ by linear interpolation between integer coordinates, successively
on each coordinate. We also let, for $i$, $j \in \{ 0,1, \ldots ,2n \}
$,
\[ D^{\circ}_{n} ( i,j ) =L_{n} ( i ) +L_{n} ( j ) -2 \max ( \check{L}_{n}
   ( i,j ) , \check{L}_{n} ( j,i ) ) +2.\mathbf{1}_{\{ \max (
   \check{L}_{n} ( i,j ) , \check{L}_{n} ( j,i ) ) <L_{n} ( i ) \wedge
   L_{n} ( j ) \}} , \qquad \]
where $\check{L}_{n} ( i,j ) = \inf \{ L_{n} ( k ) :i \leqslant k \leqslant j
\}$ if $i \leqslant j$ and $\check{L}_{n} ( i,j ) = \inf \{ L_{n} ( k ) :k \in
[ 0,j ] \cup [ i,2n ] \}$ if $i>j$. The somehow unusual indicator in this
definition only serves the purpose to match our definition of
$d^{\circ}_{\tmmathbf{q}}$; see Lemma~\ref{dno}.
% It is equal to~$1$ if and only if the three points $v_{i}$, $v_{j}$ and~$v_{\ast}$ are not aligned, in the sense that $d_{Q_{n}} ( v_{i} ,v_{j} ) > | d_{Q_{n}} (v_{i} ,v_{\ast} ) -d_{Q_{n}} ( v_{j} ,v_{\ast} ) |$.

We now recall how the mapping $\tmop{CVS}^{-1}$ is constructed. Starting from
a given plane embedding of $T_{n}$, we add the extra vertex $v_{\ast}$
arbitrarily in the unique face $f$ of the map $T_{n}$, and declare it to be
incident to a unique corner that we denote by $c_{\infty}$. Next, for every $i
\in \mathbb{Z}$ we let $s ( i ) = \inf \{ j>i:L_{n} ( j ) =L_{n} ( i ) -1
\}$, which we call the successor of $i$. Note that $s ( i ) = \infty$ if
$L_{n} ( i ) = \min  L_{n}$. The successor of the corner $c_{i}$ is then $s (
c_{i} ) =c_{s ( i )}$ by definition. The construction then consists in drawing
an arc $e_{i}$ from $c_{i}$ to $s ( c_{i} )$ for every $i \in \{ 0,1, \ldots
,2n-1 \}$, in such a way that these arcs do not cross each other, and that the
relative interior of $e_{i}$ is contained in $f$. This construction uniquely
defines a map, which is $Q_{n}$, and this map is pointed at $v_{\ast}$ (here
again, we will not specify the rooting convention). By construction, there is
a one-to-one correspondence between the corners $c_{i}$ of $T_{n}$ and the
edges $e_{i}$ of $Q_{n}$. It turns out that the natural orientation of $e_{i}$
obtained in the construction (that is, from $v_{i}$ to $v_{s ( i )}$)
coincides with the orientation that we introduced above for quadrangulations,
namely, $e_{i}$ points towards $v_{\ast}$ in $Q_{n}$. Consequently, the
oriented paths following the arcs are geodesics towards~$v_{\ast}$. See
Figure~\ref{fig:notation}.

\begin{figure}[ht]
  \centering\includegraphics{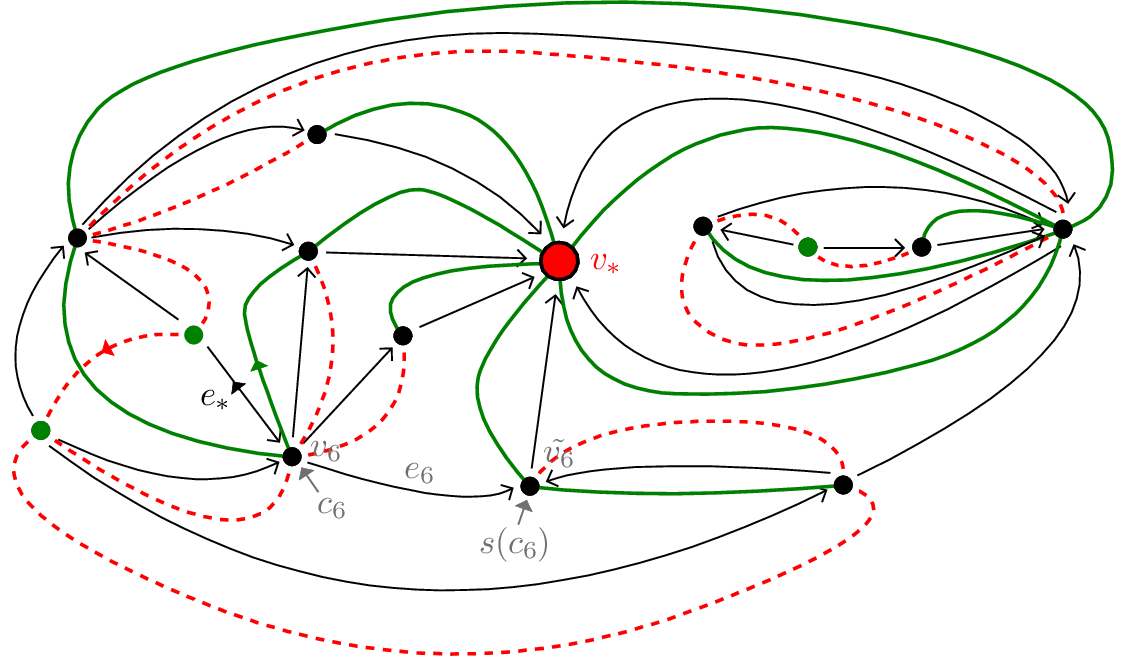}
  \caption{The two bijections, and some notation. The three green vertices
  correspond to the three faces of the map obtained by the $\tmop{AB}$
  bijection.}
	\label{fig:notation}
\end{figure}

\begin{lemma}
  \label{dno}Let $e$ be an edge of $Q_{n}$, and let $c$ be the corner of
  $T_{n}$ such that $e$ is the arc linking~$c$ with $s ( c )$. Let
  $k=d_{\tmmathbf{q}} ( e^{-} ,v_{\ast} ) -1$ and for $0 \leqslant i \leqslant
  k$ let $e_{( i )}$ be the arc going from $s^{i} ( c )$ to $s^{i+1} ( c )$.
  Then the chain $( e_{( 0 )} ,e_{( 1 )} , \ldots ,e_{( k )} )$ is the
  left-most geodesic to $v_{\ast}$ with first step $e$. Consequently,
  \[ D_{n}^{\circ} ( i,j ) =d_{Q_{n}}^{\circ} ( e_{i} ,e_{j} )  , 
     \qquad i,j \in \{ 0,1, \ldots ,2n \}  . \]
\end{lemma}

\begin{proof}
  Fix $i \in \{ 0,1, \ldots ,k \}$. By construction, every arc between
  $e_{( i )}$ and $e_{( i+1 )}$ in the clockwise order around $e_{( i
    )}^{+}$ is necessarily pointing toward $e_{( i )}^{+}$. The first
  claim easily follows. The second claim follows by noticing that the
  event $\{ \max ( \check{L}_{n} ( i,j ) , \check{L}_{n} ( j,i ) )
  <L_{n} ( i ) \wedge L_{n} ( j ) \}$ appearing in the indicator in
  the definition of $D^\circ_n$ says that neither of the left-most
  geodesic to $v_*$ with first steps $e_i$ or $e_j$ is a suffix of the
  other.
\end{proof}

We now define a function $\tilde{D}_{n}$ similar to $D_{n}$ but associated
with the map $M_{n}^{\bullet}$. Recall that $e_{i}$ is the arc of $Q_{n}$ from
the corner $c_{i}$ of $T_{n}$ to $s ( c_{i} )$. We let $\tilde{v}_{i}
=e_{i}^{+}$ so that for every $i \in \{ 0,1, \ldots ,2n \}$, $\tilde{v}_{i}$
is always an element of $V ( M_{n}^{\bullet} )$. Set
\[ \tilde{D}_{n} ( i,j ) =d_{M_{n}^{\bullet}} ( \tilde{v}_{i} , \tilde{v}_{j}
   )  , \qquad i,j \in \{ 0,1, \ldots ,2n \}  . \]
We also extend $\tilde{D}_{n}$ to a continuous function $[ 0,2n ]^{2}
\rightarrow \mathbb{R}$ as we did for $D_{n}$. Clearly, the set $\left\{
\tilde{v}_{i} :i \in \{ 0,1, \ldots ,2n \} \right\}$ is equal to
$V ( M_{n}^{\bullet} )$, so that $( \{ 0,1, \ldots ,2n \} , \tilde{D}_{n} )$
is a pseudo-metric space isometric to $( V ( M_{n}^{\bullet} )
,d_{M_{n}^{\bullet}} )$ through the mapping $i \mapsto \tilde{v}_{i}$.
Combining Corollary \ref{majoration} and Lemma \ref{dno}, we obtain the bound
\begin{equation}
  \tilde{D}_{n} ( i,j ) \leqslant D_{n}^{\circ} ( i,j ) +\Delta_{n}  ,
  \qquad i,j \in \{ 0,1, \ldots ,2n \}  , \label{boundtilde}
\end{equation}
where $\Delta_{n} \assign \Delta ( M_{n}^{\bullet} )$, and this remains true
for every $s$, $t \in [ 0,2n ]$ in place of $i$, $j$.

\subsection{Scaling limits and proof of Theorem \ref{principal}}

We now introduce renormalized versions of our encoding processes. Namely, for
$s,t \in [ 0,1 ]$, let
\[ C_{( n )} ( s ) = \frac{C_{n} ( 2n  s )}{\sqrt{2n}}  ,
   \quad L_{( n )} ( s ) = \left( \frac{9}{8n} \right)^{1/4} L_{n} ( 2n
    s )  , \quad D_{( n )} ( s,t ) = \left(
   \frac{9}{8n} \right)^{1/4} D_{n} ( 2n  s,2n  t ) \]
and define $D^{\circ}_{( n )} ( s,t )$ and $\tilde{D}_{( n )} ( s,t )$
similarly to $D_{( n )}$ by replacing $D_{n}$ with $D_{n}^{\circ}$ and
$\tilde{D}_{n}$. The main result of {\cite{legall11,miermont11}}
(which implies \eqref{eq:convQ}) shows that
one has the following convergence in distribution as $n \rightarrow \infty$ in
$\mathcal{C} ( [ 0,1 ] ,\mathbb{R} ) \times \mathcal{C} ( [ 0,1 ]
,\mathbb{R} ) \times \mathcal{C} ( [ 0,1 ]^{2} ,\mathbb{R} )$:
\begin{equation}
  ( C_{( n )} ,L_{( n )} ,D_{( n )} ) \longrightarrow ( \tmmathbf{e},Z,D )
   , \label{eq:convmap}
\end{equation}
where $( \tmmathbf{e},Z )$ is a pair of stochastic processes sometimes called
the head of the Brownian snake, and $D$ is a random pseudo-distance on $[ 0,1
]$ defined from $( \tmmathbf{e},Z )$ as follows. Define two pseudo-distances
on $[ 0,1 ]$ by the formulas
\[ d_{\tmmathbf{e}} ( s,t ) =\tmmathbf{e} ( s ) +\tmmathbf{e} ( t ) -2 \min
   \{ \tmmathbf{e} ( u ) :s \wedge t \leqslant u \leqslant s \vee t \} \]
and
\[ d_{Z} ( s,t ) =Z ( s ) +Z ( t ) -2 \max ( \check{Z} ( s,t ) , \check{Z} (
   t,s ) )  , \]
where similarly as for the definition of $D^{\circ}_{n}$ we let $\check{Z} (
s,t ) = \min \{ Z ( u ) :s \leqslant u \leqslant t \}$ if $s \leqslant t$, and
$\check{Z} ( s,t ) = \min \{ Z ( u ) :u \in [ s,1 ] \cup [ 0,t ] \}$
otherwise. Then $D$ is the largest pseudo-distance $d$ on $[ 0,1 ]$ that
satisfies the following two properties:
\begin{equation}
  \{ d_{\tmmathbf{e}} =0 \} \subset \{ d=0 \} \qquad \text{ and }
  \qquad d \leqslant d_{Z}  . \label{defBM}
\end{equation}
At this point, we recall that the Brownian map $\tmmathbf{S}$ is the quotient
space $[ 0,1 ] / \{ D=0 \}$, endowed with the (true) distance function induced
by $D$ on this set, which we still denote by $D$.

We would like to study the joint convergence of (\ref{eq:convmap}) with
$\tilde{D}_{( n )}$, and show that the limit of the latter is $D$ as well. To
this end, we proceed in three steps.

\paragraph{First step: tightness}We observe that (\ref{eq:convmap}) implies
that $D_{( n )}^{\circ}$ converges (jointly) to $d_{Z}$. On the other hand,
the bound (\ref{boundtilde}) combined with Lemma \ref{oplogn} easily implies
that the laws of $\tilde{D}_{( n )}$, $n \geqslant 1$, \ form a relatively
compact family of probability measures on $\mathcal{C} ( [ 0,1 ]^{2}
,\mathbb{R} )$, by repeating the argument of {\cite{legall06}}. Indeed, for
every $\delta >0$, let
\[ \omega ( \tilde{D}_{( n )} , \delta ) = \sup \left\{ | \tilde{D}_{( n )} (
   s,t ) - \tilde{D}_{( n )} ( s' ,t' ) | : | s-s' | \vee | t-t' |
   \leqslant \delta \right\} \]
be the modulus of continuity of $\tilde{D}_{( n )}$ evaluated at $\delta$, so
by the triangle inequality and (\ref{boundtilde}), we have
\begin{align*}
  \omega ( \tilde{D}_{( n )} , \delta ) & \leqslant 2 \sup   \left\{ \tilde{D}_{( n
  )} ( s,s' ) : | s-s' | \leqslant \delta \right\}\\
  & \leqslant 2 \sup \left\{ D_{( n )}^{\circ} ( s,s' )  : | s-s' |
  \leqslant \delta \right\} + \frac{\Delta_{n}}{( 8n/9 )^{1/4}}  .
\end{align*}
It follows from Lemma \ref{oplogn} and the convergence in distribution
(\ref{eq:convmap}) that for every $\varepsilon >0$,
\[ \limsup_{n \rightarrow \infty}  P \big( \omega ( \tilde{D}_{( n )} , \delta )
   \geqslant \varepsilon \big) \leqslant P \big( 2 \sup \left\{ d_{Z} ( s,s' )
   : | s-s' | \leqslant \delta \right\} \geqslant \varepsilon \big) 
\]
and the a.s.\ continuity of $Z$ implies that this converges to $0$ as $\delta
\rightarrow 0$. Since $\tilde{D}_{( n )} ( 0,0 ) =0$, this entails the
requested tightness result.

Hence, up to extraction of a subsequence $( n_{k} )$, we may assume that
\begin{equation}
  \big( C_{( n )} ,L_{( n )} ,D^{ \circ}_{( n )} ,D_{( n )} ,
  \tilde{D}_{( n )} ,n^{-1/4} \Delta_{n} \big) \longrightarrow \big(
  \tmmathbf{e},Z,d_{Z} ,D, \tilde{D} ,0 \big)  \label{eq:convmapbis}
\end{equation}
in distribution, where $\tilde{D}$ is a random continuous function on $[ 0,1
]^{2}$. In order to simplify the arguments to follow, we apply the Skorokhod
representation theorem, and assume that \emph{the underlying probability
space is chosen so that this convergence holds almost surely rather than in
distribution}. Until the end of the paper, all the convergences as $n
\rightarrow \infty$ are understood to take place along this subsequence $(
n_{k} )$.

\paragraph{Second step: bound on $\tilde{D}$}It is not difficult to check
that $\tilde{D}$ is a pseudo-distance, because $\tilde{D}_{( n )}$ is already
symmetrical and satisfies the triangle inequality, and because $\tilde{D}_{( n
)} ( s,s ) =0$ as soon as $s$ is in $\left\{ k/2n :k \in \{ 0,1, \ldots ,2n
\} \right\}$. Let us prove that $\tilde{D}$ satisfies the properties appearing
in~(\ref{defBM}). First, assume that $d_{\tmmathbf{e}} ( s,t ) =0$. Then it is
elementary to see that there are sequences of integers $i_{n} $,
$j_{n}$ such that $i_{n} /2n$ and $j_{n} /2n$ respectively converge to $s$ and
$t$, and such that $v_{i_{n}} =v_{j_{n}}$. As a consequence, it holds that
$\tilde{v}_{i_{n}}$ and $\tilde{v}_{j_{n}}$ lie in the same face or in two
adjacent faces of $M_{n}^{\bullet}$, and therefore are at distance at most
$\Delta_{n}$ in $M_{n}^{\bullet}$. Consequently, one has that
\[ \tilde{D} ( s,t ) = \lim_{n \rightarrow \infty} \tilde{D}_{( n )} \left(
   \frac{i_{n}}{2n} , \frac{j_{n}}{2n} \right) = \lim_{n \rightarrow \infty}  
   \left( \frac{9}{8n} \right)^{1/4} d_{M_{n}^{\bullet}} ( \tilde{v}_{i_{n}} ,
   \tilde{v}_{j_{n}} ) =0 , \]
as wanted. Finally, the bound $\tilde{D} \leqslant d_{Z}$ is a simple
consequence of (\ref{boundtilde}) and (\ref{eq:convmapbis}).

From this and the definition of $D$ as the largest pseudo-distance satisfying
(\ref{defBM}), we obtain that $\tilde{D} \leqslant D$. On the other hand, let
$s_{\ast}$ be the (a.s.\ unique~{\cite{legweill}}) point at which $Z$ attains
its minimum. Taking a sequence $( i_{n} )$ such that $\tilde{v}_{i_{n}}
=v_{\ast}$, it is not difficult to see, using the convergence of $L_{( n )}$
to $Z$, that $i_{n} /2n$ must converge to $s_{\ast}$. Therefore, by choosing
other sequences $( j_{n} )$ such that $j_{n} /2n$ converges, it follows
from~{\eqref{eq:samegeod}} that, almost surely,
\begin{equation}
  \tilde{D} ( s_{\ast} ,s ) =D ( s_{\ast} ,s ) =Z_{s} -Z_{s_{\ast}}
  \qquad   \text{for every }   s \in [ 0,1 ]  .
  \label{eq:distroot}
\end{equation}
\paragraph{Third step: re-rooting argument}The final crucial property on which
the proof relies is that if $U_{1}$, $U_{2}$ are independent random variables
in $[ 0,1 ]$ that are also independent of all the previously considered random
variables, then
\begin{equation}
  \tilde{D} ( U_{1} ,U_{2} ) \overset{( d )}{=} \tilde{D} (s_{\ast} ,U_{1} )
   . \label{eq:eqdist}
\end{equation}
The proof of this re-rooting identity is a bit long so that we postpone it to
the next Section. Let us see how this concludes the proof of Theorem
\ref{principal}. Observe that~$D$ also satisfies property (\ref{eq:eqdist})
(which can be obtained using the fact that quadrangulations are invariant
under re-rooting, see {\cite{legall11}}). Given this and (\ref{eq:distroot}),
we deduce that
\[ E \big[ \tilde{D} ( U_{1} ,U_{2} ) \big] =E \big[ \tilde{D} ( s_{\ast} ,U_{1} ) \big] =E \big[
   D ( s_{\ast} ,U_{1} ) \big] =E \big[ D ( U_{1} ,U_{2} ) \big]  , \]
which entails that $\tilde{D} ( U_{1} ,U_{2} ) =D ( U_{1} ,U_{2} )$ a.s.,
since we already know that $\tilde{D} \leqslant D$. By Fubini's theorem, this
shows that a.s.\ $\tilde{D}$ and $D$ agree on a dense subset of $[ 0,1 ]^{2}$,
hence everywhere by continuity. The convergence (\ref{eq:convmapbis}) can thus
in part be rewritten
\begin{equation}
  \label{eq:1}
 \big( C_{( n )} ,L_{( n )} ,D_{( n )} , \tilde{D}_{( n )} \big)
\underset{n \rightarrow \infty}{\overset{( d )}{\longrightarrow}}
\big( \tmmathbf{e},Z,D,D \big) , 
\end{equation}
from which it is easy to deduce Theorem~\ref{sec:main-results}, using the
fact that the Gromov--Hausdorff distance between two metric spaces is
bounded by the distortion of any correspondence between these spaces,
see for instance Section 3.3 in {\cite{LGMi11}}. Namely, we can assume
again that \eqref{eq:1} holds almost surely rather than in distribution by a
further use of the Skorokhod representation theorem. Then, if we
denote by $\mathbf{p}:[0,1]\to \tmmathbf{S}$ the canonical projection,
we note that the sets $\left\{ ( v_{\lfloor 2nt\rfloor} ,\mathbf{p}(t)
  ) :t\in [0,1] \right\}$ and $\left\{ (\tilde{v}_{\lfloor
    2nt\rfloor},\mathbf{p}(t) ) :t\in [0,1]\right\}$ are
correspondences between, on the one hand, the metric spaces $( V (
Q_{n} )\setminus\{v_*\} ,(9/8n)^{1/4}d_{Q_{n}} )$ and $( V (
M_{n}^{\bullet} ) ,(9/8n)^{1/4}d_{M_{n}^{\bullet}} )$, and the
Brownian map $(\tmmathbf{S},D)$ on the other hand. Moreover, their
distortions are bounded from above by
\[ \underset{s,t \in [ 0,1 ]}{\sup} \big| D_{(n)} ( \lfloor
2ns\rfloor/2n,\lfloor 2nt\rfloor/2n ) - D(s,t)\big|\quad \mbox{ and
}\quad \underset{s,t \in [ 0,1 ]}{\sup}\big| \tilde{D}_{(n)} ( \lfloor
2ns\rfloor/2n,\lfloor 2nt\rfloor/2n ) - D(s,t)\big|\, , \] which both
converge to $0$ almost surely. This, and the obvious fact that the
Gromov-Hausdorff distance between $( V ( Q_{n} ),d_{Q_{n}} )$ and $( V
( Q_{n} )\setminus\{v_*\} ,d_{Q_{n}} )$ is at most $1$, imply
Theorem~\ref{sec:main-results}. Corollary~\ref{principal} follows by
Proposition~\ref{TV}.

\section{Proof of the re-rooting identity}\label{reroot}

It remains to prove (\ref{eq:eqdist}). This again relies on a limiting
argument. Namely, recall that the distinguished point $v_{\ast}$ in
$M_{n}^{\bullet}$ is a uniformly chosen element of $V ( M_{n}^{\bullet} )$.
Therefore, if $V_{1}$ and $V_{2}$ are two other such elements, chosen
independently, and independently of $v_{\ast}$, then it holds trivially that
\[ d_{M_{n}^{\bullet}} ( V_{1} ,V_{2} ) \overset{( d )}{=}
   d_{M_{n}^{\bullet}} ( v_{\ast} ,V_{1} )  . \]
On the other hand, let $( i_{n} )$ be a sequence of integers such that
$\tilde{v}_{i_{n}} =v_{\ast}$, so that $i_{n} /2n \rightarrow s_{\ast}$. If
$U_{1} ,U_{2}$ are uniform on $[ 0,1 ]$ as above, then they naturally code
the vertices $\tilde{v}_{\lfloor 2n  U_{1} \rfloor} ,
\tilde{v}_{\lfloor 2n  U_{2} \rfloor}$, and so by
{\eqref{eq:convmapbis}} we have that
\[ \left( \frac{9}{8n} \right)^{1/4}\! d_{M_{n}^{\bullet}} ( v_{\ast} ,
   \tilde{v}_{\lfloor 2n  U_{1} \rfloor} ) \underset{n \rightarrow
   \infty}{\longrightarrow} \tilde{D} ( s_{\ast} ,U_{1} ) \quad
   \text{and} \quad \left( \frac{9}{8n} \right)^{1/4}\!
   d_{M_{n}^{\bullet}} ( \tilde{v}_{\lfloor 2n  U_{1} \rfloor} ,
   \tilde{v}_{\lfloor 2n  U_{2} \rfloor} ) \underset{n \rightarrow
   \infty}{\longrightarrow} \tilde{D} ( U_{1} ,U_{2} )  . \]
Therefore, (\ref{eq:eqdist}) would follow directly if the vertices
$\tilde{v}_{\lfloor 2n  U_{1} \rfloor}$ and $\tilde{v}_{\lfloor 2n
 U_{2} \rfloor}$ were uniform in $V ( M_{n}^{\bullet} )$.
Unfortunately, the probability that $\tilde{v}_{\lfloor 2n  U_{1}
\rfloor}$ is equal to a given vertex $v$ of $M_{n}^{\bullet}$ is proportional
to the number of edges~$e$ of $Q_{n}  $pointing towards $v_{\ast}$ such that
$e^{+} =v$. Using the construction of the $\tmop{AB}$
bijection, one can see that this number of edges is precisely the degree of $v$ in
$M_{n}^{\bullet}$, but we leave this as an exercise to the reader as we are
not going to use it explicitly. \

On the other hand, (\ref{eq:eqdist}) will follow if $\tilde{v}_{\lfloor 2n
 U_{1} \rfloor}$ can be coupled with a uniformly chosen vertex $V_{1}$
in $M_{n}^{\bullet}$ in such a way that $d_{M_{n}^{\bullet}} (
\tilde{v}_{\lfloor 2n  U_{1} \rfloor} ,V_{1} ) =o ( n^{1/4} )$ almost
surely, possibly along a subsequence of $( n_{k} )$. This is what we now
demonstrate, except that the vertex $V_{_{1}}$ that we will produce (denoted
by $v_{j_{n}}$ below) will be uniform on $V ( M_{n}^{\bullet} ) \setminus \{
v_{\ast} \}$ rather than on $V ( M_{n}^{\bullet} )$. This distinction is of
course of no importance.

First recall that $V ( M_{n}^{\bullet} ) =V ( Q_{n} ) \setminus V_{\max} (
Q_{n} )$ where $V_{\max} ( Q_{n} )$ was defined in Section \ref{sec:CVSAB} as
the set of vertices of $Q_{n}$ whose neighbors are all closer to $v_{\ast}$.
With the usual identification of vertices of $V ( Q_{n} ) \setminus \{
v_{\ast} \}$ with $V ( T_{n} )$, we can view the vertices $V_{\max} ( Q_{n} )$
as a subset of $V ( T_{n} )$.

\begin{lemma}
  \label{vmax}A vertex $v \in V ( T_{n} )$ is an element of $V_{\max} ( Q_{n}
  )$ if and only if its label is a local maximum in $T_{n}$ in the broad
  sense. Namely, for every vertex $u$ adjacent to $v$ in $T_{n}$, it holds
  that $\ell_{n} ( u ) \leqslant \ell_{n} ( v )$.
\end{lemma}

\begin{proof}
  Let $ l= \ell_{n} ( v )$. Assume first that one of the neighbors $u$ of
  $v$ has a label $l+1$. Let $c$ be the last corner of $u$ before visiting $v$
  in contour order. Then the successor $s ( c )$ in the $\tmop{CVS}$ bijection is by
  construction a corner incident to $v$, so that $u$ and $v$ are adjacent in
  $Q_{n}$, but $u$ is further away from $v_{\ast}$ than $v$, so that $v \nin
  V_{\max} ( Q_{n} )$. Conversely, if a vertex $u$ adjacent to $v$ has label
  $l$ or $l-1$, consider the maximal subtree of $T_{n}$ that contains $u$ but
  not $v$. Then clearly every corner incident to a vertex in this subtree with
  label $l+1$ cannot be linked by an arc to $v$. Moreover, by construction,
  every corner of $v$ is linked to a vertex with label $l-1$. So if $v$ is a
  local maximum in $T_{n}$ in the broad sense, $v$ has no neighbors in $Q_{n}$
  that are further away from $v_{\ast}$ than $v$, so $v \in V_{\max} ( Q_{n}
  )$. 
\end{proof}

If $( \tmmathbf{t},\tmmathbf{l} )$ is a labeled tree, we will let $V_{\max} (
\tmmathbf{t},\tmmathbf{l} )$ be the set of vertices of $\tmmathbf{t}$ that are
local maxima of~$\tmmathbf{l}$ in the broad sense, so the last lemma states
that $V_{\max} ( Q_{n} ) =V_{\max} ( T_{n} , \ell_{n} )$.

Now let $N_{0} =0$ and, for $j \in \{ 1,2, \ldots ,2n \}$, let $N_{j}$ be the
number of vertices in $\{ v_{0} ,v_{1} , \ldots ,v_{j-1} \}$ that do not
belong to $V_{\max} ( T_{n} , \ell_{n} )$. Note that $N_{2n} =\#V ( T_{n} )
-\#V_{\max} ( T_{n} , \ell_{n} ) =\#V ( M_{n}^{\bullet} ) -1$ (the~$-1$ comes
from the fact that $V ( T_{n} ) =V ( M_{n}^{\bullet} ) \backslash \{ v_{\ast}
\}$). Fix $t \in [ 0,1 ]$ and let $i= \lfloor 2nt \rfloor$. Let also $v ( 0 )
,v ( 1 ) , \ldots ,v ( h ) =v_{i}$ be the spine consisting of the ancestors of
$v_{i}$ in $T_{n}$ indexed by their heights, so that $v ( 0 ) =v_{0}$ is the
root vertex of $T_{n}$ and $h=C_{n} ( i )$ is the height of~$v_{i}$. Note that
the vertices $v_{0} ,v_{1} , \ldots ,v_{i-1} ,v_{i}$ are the vertices
contained in the subtrees of $T_{n}$ rooted on $v ( 0 ) ,v ( 1 ) , \ldots ,v (
h )$ that lie to the left of the spine, and more specifically, between the
root corner $c_{0}$ and the corner $c_{i}$ of $T_{n}$. We let $T ( 0 ) ,T ( 1
) , \ldots ,T ( h )$ be these trees, ordered by size, that is, in such a way
that $n_{0} \geqslant n_{1} \ldots \geqslant n_{h}$ where $n_{j} =\#E ( T ( j
) )$ (we arbitrarily choose in case of ties). Note that $T ( j )$ is naturally
rooted at the first corner of a vertex $v ( k_{j} )$ visited by the contour
exploration of~$T_{n}$. For $j>h$, we set $n_{j} =0$.

We also let $L_{j}$ be the label function $\ell_{n}$ restricted to $T ( j
)$, and shifted by the label of the root, so that $L_{j} ( u ) = \ell_{n} (
u ) - \ell_{n} ( v ( k_{j} ) )$ for $u \in V ( T ( j ) )$.

We then note two important facts:
\begin{enumerate}
  \item Conditionally given $( n_{0} ,n_{1} , \ldots )$, the labeled trees $(
  T ( 0 ) ,L_{0} ) , ( T ( 1 ) ,L_{1} ) , \ldots , ( T ( h ) ,L_{h} )$ are
  independent uniform elements of $\mathbb{T}_{n_{0}} ,\mathbb{T}_{n_{1}} ,
  \ldots ,\mathbb{T}_{n_{h}}$ respectively, where $h= \max \left\{ i :
  n_{i} >0 \right\}$.
  
  \item For every $\varepsilon >0$, there exists $K>0$ such that, for
  sufficiently large $n$, $P ( n_{0} +n_{1} + \cdots +n_{K} <n  ( t-
  \varepsilon ) ) < \varepsilon$.
\end{enumerate}
The first property is easy. To see why the second is true, note that the
contour processes of $T ( 0 ) ,T ( 1 ) , \ldots ,T ( h )$ are the excursions
of $( C_{n} ( s ) ,0 \leqslant s \leqslant i )$ above the process $( \inf \{
C_{n} ( u ) :s \leqslant u \leqslant i \} ,0 \leqslant s \leqslant i )$. The
convergence of the rescaled contour function $C_{( n )}$ to the normalized
Brownian excursion $\tmmathbf{e}$ then easily implies that for every $j
\geqslant 0$, the $j$+1-th longest of these excursions (the one coding $T ( j
)$) converges uniformly to the $j+1$-th longest excursion of
$\tmmathbf{e}$ above the process $( \inf \{ \tmmathbf{e} ( u ) :s
\leqslant u \leqslant t \} ,0 \leqslant s \leqslant t )$. Note that this
excursion is unambiguously defined. This implies that $n_{j} /n$ converges to
the length of the $j+1$-th longest excursion. By standard properties of
Brownian motion, these excursion lengths sum to $t$, and this implies the
wanted result.

Now since the label functions $L_{j}$ are just shifted versions of
$\ell_{n}$, note that
\[ \left| N_{i} - \sum_{j=0}^{h} \Gamma_{j} \right| \leq h, \] where
$\Gamma_{j} \assign \#V(T(j)) - \#V_{\max} ( T ( j ) ,L_{j} )$. Since
$h=C_{n} ( i )$ converges after renormalization by $\sqrt{2n}$ to
$\tmmathbf{e} ( t )$, we obtain that $h/n$ converges to $0$ in
probability. Also, conditionally given $n_{1} ,n_{2} , \ldots$, point
1.\ above implies that the random variables $\Gamma_{j}$, $j \geqslant
0$, are independent and, by Lemma \ref{vmax}, $\Gamma_{j}$ has the
same distribution as $V ( M_{n_{j}}^{\bullet} ) -1$. But the $L^{2}$
convergence of $2 \, \#V ( M_{n} ) /n$ to $1$ established in the proof
of Proposition~\ref{TV} entails that $2 ( \#V ( M_{n}^{\bullet} ) -1 )
/n$ also converges to $1$ in probability, by Proposition~\ref{TV}. Fix
$\varepsilon >0$, $K$ as in point 2.\ above, and $N$ such that $n \geq
N$ implies that both the conclusion of point~2.{\quad}and $P ( | 2 (
\#V ( M_{n}^{\bullet} ) -1 ) /n-1 | > \varepsilon /t ) < \varepsilon /
( K+1 )$ hold. Observe that if both $\sum_{j=0}^{K} n_{j} \geqslant n
( t- \varepsilon )$ and $\sum^{K}_{j=0} n_{j} ( 1- \varepsilon /t )
\leqslant 2 \, \sum_{j=0}^{K} \Gamma_{j} \leqslant \sum^{K}_{j=0}
n_{j} ( 1+ \varepsilon /t )$ hold, then, on the one hand, $2
\sum_{j=0}^{h} \Gamma_{j} \geqslant 2 \, \sum_{j=0}^{K} \Gamma_{j}
\geqslant n \, ( t-2 \varepsilon )$ and, on the other hand, $2
\sum_{j=0}^{h} \Gamma_{j} \leqslant 2 \sum_{j=0}^{K} \Gamma_{j} +2
\sum^{h}_{j=K+1} n_{j} \leqslant n \, ( t+2 \varepsilon )$, because it
always holds that $\Gamma_{j} \leqslant n_{j}$ and $\sum_{j=0}^{h}
n_{j} \leqslant n t$. As a result,

\begin{align*}
  P\left(\Bigg|{\frac{2}{n}}\sum_{j=0}^{h}{\Gamma}_{j}-t\Bigg|{\geqslant} 2 {\varepsilon} \right) 
  & \leqslant P\left(\sum_{j=0}^{K}n_{j}<n
  (t-{\varepsilon})\right)+\sum_{j=0}^{K}P\left(\left|{\frac{2\,{\Gamma}_{j}}{n_{j}}}-1\right|>{\frac{{\varepsilon}}{t}}\right)\\
  & \leqslant 2{\varepsilon}+(K+1)P(n_{K}<N).
\end{align*}
The last inequality is obtained by conditioning on $n_{j}$ and treating
separately whether $n_{j} \geq N$ or $n_{j} <N$. As $n_{K} /n$ converges to a
non-degenerate random variable, it follows that $2N_{i} /n$ converges in
probability to $t$.

Since this is valid for every $t \in [ 0,1 ]$, standard monotony arguments
entail that
\[ \left( \frac{2N_{\lfloor 2nt \rfloor}}{n} ,0 \leqslant t \leqslant 1
   \right) \underset{n \rightarrow \infty}{\longrightarrow} \tmop{Id}
   _{[ 0,1 ]} . \]
in probability for the uniform norm. Upon further extraction from $( n_{k} )$,
we can in fact assume that this convergence holds a.s.

Now let $U_{1}$ be uniform in $[ 0,1 ]$ as above, and let $j_{n}$ be the first
integer $j$ such that $N_{j} >U_{1} \times N_{2n}$. By definition, the vertex
$v_{j_{n}}$ is uniformly distributed in $V(T_n)\setminus V_{\max} ( T_{n} , \ell_{n} ) =V
( M_{n}^{\bullet} ) \setminus \{ v_{\ast} \}$. On the other hand, the previous
convergence implies that $j_{n} /2n \rightarrow U_{1}$. Consequently, since
$v_{j_{n}}$ is at distance at most $\Delta_{n}$ from $\tilde{v}_{j_{n}}$ in
$M_{n}$,
\[ \left( \frac{9}{8n} \right)^{1/4} d_{M_{n}^{\bullet}} ( v_{j_{n}} ,
   \tilde{v}_{\lfloor 2nU_{1} \rfloor} ) \leq   \left( \frac{9}{8n}
   \right)^{1/4}  \big( \tilde{D}_{n} ( j_{n} , \lfloor 2nU_{1} \rfloor ) +
   \Delta_{n} \big) \longrightarrow \tilde{D} ( U_{1} ,U_{1} ) =0  , \]
where the last convergence comes from (\ref{eq:convmapbis}), and this is what
we needed to conclude.

\bibliographystyle{alpha}
\bibliography{biblio}
\end{document}